\documentclass[12pt]{amsart}

\usepackage{latexsym,amsmath,amssymb,amsthm,amsfonts}

\usepackage[lite,alphabetic]{amsrefs}

\newcommand{\ds}{\displaystyle}
\newcommand{\R}{{\mathbb{R}}}
\newcommand{\M}{{\mathcal{M}}}
\newcommand{\A}{{\mathcal{A}}}
\newcommand{\B}{{\mathcal{B}}}

\newcommand{\mA}{{\mathfrak{A}}}
\newcommand{\wt}{\widetilde}

\renewcommand{\phi}{\varphi}

\def\be{\begin{equation}}
\def\ee{\end{equation}}

\newtheorem{theorem}{Theorem}[section]
\newtheorem{lemma}[theorem]{Lemma}
\newtheorem{corollary}[theorem]{Corollary}
\theoremstyle{remark}
\newtheorem{remark}[theorem]{Remark}

\numberwithin{equation}{section}

\title{On concentrators and related approximation constants}

\author{A. V. Bondarenko}
\address{Centre de Recerca Matem\`{a}tica
Apartat 50, 08193 Bellaterra, Barcelona, Spain \newline and\newline
Department of Mathematical Analysis, National Taras Shevchenko
University, str. Vo\-lo\-dy\-myr\-ska, 64, Kyiv, 01033, Ukraine}
\email{andriybond@gmail.com}

\author{A. Prymak}
\address{Department of Mathematics, University of Manitoba, Winnipeg, MB, R3T2N2, Canada}
\email{prymak@gmail.com}

\author{D. Radchenko}
\address{Max Planck Institute for Mathematics, Vivatsgasse 7, 53111 Bonn,
Germany\newline and\newline Department of Mathematical Analysis,
National Taras Shevchenko University, str. Vo\-lo\-dy\-myr\-ska, 64,
Kyiv, 01033, Ukraine} \email{danradchenko@gmail.com}

\thanks{The first author thanks the Mathematisches Forschungsinstitut Oberwolfach for their hospitality during the preparation
of this manuscript and for providing a stimulating atmosphere for research. The second author was supported by NSERC of Canada, including the visit to Centre de Recerca Matem\`{a}tica (Barcelona) in April 2012.}



\begin{document}

\begin{abstract}
Pippenger (\cite{Pi77}) showed the existence of $(6m,4m,3m,6)$-concentrator for each positive integer $m$ using a probabilistic method. We generalize his approach and prove existence of $(6m,4m,3m,5.05)$-concentrator (which is no longer regular, but has fewer edges). We apply this result to improve the constant of approximation of almost additive set functions by additive set functions from $44.5$ (established by Kalton and Roberts in~\cite{KaRo83}) to $39$. We show a more direct connection of the latter problem to the Whitney type estimate for approximation of continuous functions on a cube in $\R^d$ by linear functions, and improve the estimate of this Whitney constant from $802$ (proved by Brudnyi and Kalton in~\cite{BrKa00}) to $73$.
\end{abstract}

\maketitle

\section{Introduction}

Our original motivation was the following Whitney-type inequality, valid for each $f\in C([0,1]^d)$:
\[
\min_{L}\max_{x\in [0,1]^d}|f(x)-L(x)|\le w_2(d) \max_{x,y\in[0,1]^d}|f(x)+f(y)-2f((x+y)/2)|,
\]
where the minimum is taken over all polynomials $L$ in $d$ variables of total degree $\le 1$ (linear polynomials), and $C([0,1]^d)$ is the space of all continuous real-valued functions on the unit cube $[0,1]^d$. Brudnyi and Kalton (see~\cite{BrKa00}) showed that $w_2(d)\le 802$ and conjectured that $w_2(d)\le 2$. We will show here that $w_2(d)\le 73$, and improve some other constants along the way.

The above estimates, however, stem from seemingly irrelevant combinatorial problem of existence of certain concentrators. An $(m,p,q,r)$-concentrator is a bipartite graph with $m$ inputs and $p$ outputs, not more than $mr$ edges, such that for any set of $k\le q$ inputs, there exist $k$ disjoint edges to some $k$ outputs. Using a probabilistic argument, Pippenger~\cite{Pi77} showed that $(6m,4m,3m,6)$-concentrators exist for any integer $m\ge 1$. Reducing the average degree of inputs for large $m$ is of primary interest in our context. Our main result is the following theorem.
\begin{theorem}\label{conc}
For any large enough integer $m$ there exists a $(6m,4m,3m,5.05)$-concentrator.
\end{theorem}
For the proof, we use a modification of Pippenger's approach, but this requires much more technical estimates. Unfortunately, our method does not allow to prove that $(6m,4m,3m,5)$-concentrators exist for large $m$, but we conjecture that this is so, see Remark~\ref{r2.2}.


Pippenger's concentrators were used by Kalton and Roberts in~\cite{KaRo83} to prove the following. There exists an absolute constant $K\le 44.5$ such that for any algebra $\mA$ of finite sets and any map $\nu:\mA\to\R$ satisfying $|\nu(A \cup B)-\nu(A)-\nu(B)|\le 1$ whenever $A\cap B=\emptyset$, there exists an additive set-function $\mu:\mA\to\R$ (i.e., $\mu(A \cup B)=\mu(A)+\mu(B)$ for $A\cap B=\emptyset$), satisfying $|\nu(A)-\mu(A)|\le K$ for any $A\in \mA$. We remark that the same is true if one does not restrict the elements of $\mA$ to be \emph{finite} sets, see~\cite[Proof of Theorem~4.1, p.~809]{KaRo83}. From Theorem~\ref{conc}, we immediately obtain the following improvement.
\begin{corollary}\label{Kalton}
$K< 39$.
\end{corollary}

Since Brudnyi and Kalton~\cite{BrKa00} reduced the problem of estimating $w_2(d)$ to the problem of estimating $K$, Corollary~\ref{Kalton} would provide an immediate (but insignificant) improvement of the estimate on $w_2(d)$. We establish a more direct connection between these two questions and prove the following.
\begin{theorem}\label{Brudnyi}
$w_2(d)< 73$.
\end{theorem}

Using Corollary~\ref{Kalton} and Theorem~\ref{Brudnyi}, one can follow~\cite{BrKa00} to obtain an improvement of other approximation constants, including Whitney constant for unit balls of finite dimensional $l_p$-spaces, homogeneous Whitney constants, etc.

The paper is organized as follows. In Section~\ref{sec-conc}, we state the main technical lemma and use it to prove Theorem~\ref{conc}. The lemma itself is proved in Section~\ref{tech} using reduction to a non-linear optimization problem, which was resolved with the aid of a computer. The proof of Corollary~\ref{Kalton} and Theorem~\ref{Brudnyi} can be found in Section~\ref{sec-const}.

\section{Concentrators}\label{sec-conc}

Let $\binom nm=\frac{n!}{m!(n-m)!}$ be the binomial coefficient, and we set $\binom nm=0$ if $m<0$ or $m>n$. The most technical part of our result is the following lemma, which will be proved later in Section~\ref{tech}.
\begin{lemma}\label{lemma-tech}
For any large integer $m$, with $s=\lceil 5.7 m \rceil$, we have
\be\label{main}
\sum_{k=1}^{3m} \sum_{l=0}^{k} \sum_{r=0}^{k} \binom sl \binom{6m-s}{k-l} \binom {s-4m}{r} \binom{8m-s}{k-r} \frac{\binom{8k-r}{6k-l}}{\binom{36m-s}{6k-l}}<1.
\ee
\end{lemma}

Now we show how Lemma~\ref{lemma-tech} implies our main result closely following the idea of~\cite{Pi77} with some extra necessary calculations appearing from non-regularity of the graph.
\begin{proof}[Proof of Theorem~\ref{conc}]
Let $s=\lceil 5.7 m \rceil$, $N:=36m-s$, and $\M:=\{0,1,\dots,N-1\}$. Any permutation $\pi$ on $\M$ defines a bipartite graph $G(\pi)$ with inputs $\{0,1,\dots,6m-1\}$ and outputs $\{0,1,\dots,4m-1\}$, where for every $x\in\M$ there is an edge from $(x \mod 6m)$ to $(\pi(x) \mod 4m)$. There are $6m-s$ inputs of degree $6$ and $s$ inputs of degree $5$; $s-4m$ outputs of degree $7$ and $8m-s$ outputs of degree $8$. Total average degree of the inputs is at most $\frac{36m-5.7m}{6m}=5.05$.

Following Pippenger, we want to compute the probability that a random (with respect to the uniform distribution) permutation $\pi$ is ``bad'', that is for some $k$, $1\le k\le 3m$, there exists a set $A$ of $k$ inputs and a set $B$ of $k$ outputs in $G(\pi)$ such that every edge out of $A$ goes into $B$. Let $l$, $0\le l\le k$, be the number of vertices from $A$ that have degree $5$, and let $r$, $0\le r\le k$, be the number of vertices from $B$ that have degree $7$. Then $A$ corresponds to a set $\A$ of $6(k-l)+5l=6k-l$ elements from $\M$, and $B$ corresponds to a set $\B$ of $8(k-r)+7r=8k-r$ elements from $\M$. Note that $\A$ can be chosen in $\binom sl \binom{6m-s}{k-l}$ ways, while $\B$ can be chosen in $\binom {s-4m}{r} \binom{8m-s}{k-r}$ ways, which is reflected in the first four factors of~\eqref{main} (for some values of $k$ and $r$ one or more of these binomial coefficients may be zero). The probability that a permutation $\pi$ sends each element of $\A$ into $\B$ is equal to
\[
(8k-r)(8k-r-1)\dots((8k-r)-(6k-l)+1)\frac{(N-(6k-l))!}{N!}=\frac{\binom{8k-r}{6k-l}}{\binom{N}{6k-l}}=\frac{\binom{8k-r}{6k-l}}{\binom{36m-s}{6k-l}}.
\]
This shows that the probability that a permutation is ``bad'' is bounded by the left-hand side of~\eqref{main}, and by Lemma~\ref{lemma-tech}, it is bounded by one. Hence, a ``good'' permutation exists, and the existence of the required concentrator is proved.
\end{proof}

\begin{remark}\label{r2.2}
Essentially, \cite{Pi77} considers the case of $s=0$, and here we find the largest possible $s$ permitting generalization. It is easy to see from the proof of Theorem~\ref{conc}, that if~\eqref{main} is satisfied with $s=6m$, then a $(6m,4m,3m,5)$-concentrator exists. Let $s(m)$ be the largest value of $s$ so that~\eqref{main} is satisfied. For small values of $m$, the quotient $s(m)/m$ appears to be larger, and in fact, computer computations show that $s(m)/m\ge 6$ for all $m\le150$ (but not for $m=151$). However, as $m\to\infty$, we have $s(m)/m\to c^*\approx 5.72489$, see Remark~\ref{c**}. Hence, our refinement of Pippenger's probabilistic approach allows to prove asymptotic existence of $(6m,4m,3m,5.05)$-concentrators, but does not imply the existence of $(6m,4m,3m,5)$-concentrators for large $m$. We conjecture that $(6m,4m,3m,5)$-concentrators \emph{do} exist for large $m$, since our method shows that a random graph from certain configuration space will provide ``almost'' the required concentrator. If an ``average'' object is ``almost good'', it is reasonable to expect that some ``best'' object will be ``good'', but the proof may require a completely different, and, perhaps, non-probabilistic approach.
\end{remark}

\section{Constants}\label{sec-const}
\begin{proof}[Proof of Corollary~\ref{Kalton}]
Following the proof of~\cite[Theorem~4.1, p.~811]{KaRo83}, we see that if $(6m,4m,3m,\gamma)$-concentrators exists for large enough $m$, then
\[
K\le\frac{7+4\gamma-4/3}{2/3}.
\]
For $\gamma=5.05$, we obtain $K\le 38.8<39$.
\end{proof}

The following lemma is a slight modification of~\cite[Theorem~4.1]{KaRo83} combined with new concentrators, which uses a stronger condition on the function being approximated and achieves a better constant.
\begin{lemma}\label{better}
For any algebra $\mA$ of sets and any map $\nu:\mA\to\R$ satisfying
\be\label{nuestimate}
|\nu(A)+\nu(B)-\nu(A \cap B)-\nu(A \cup B)|\le 1 \quad \text{for any }A,B\in\mA,
\ee
and $\nu(\emptyset)=0$, there exists an additive set-function $\mu:\mA\to\R$, satisfying $|\nu(A)-\mu(A)|\le \wt K$ for any $A\in \mA$, where $\wt K< 36$.
\end{lemma}
\begin{proof}
Note that when $\nu(\emptyset)=0$, the condition~\eqref{nuestimate} implies $|\nu(A)+\nu(B)-\nu(A \cup B)|\le 1$ for any $A\cap B=\emptyset$. Therefore, we can follow the proof of~\cite[Theorem~4.1]{KaRo83} verbatim with a small change that will be described now. Below $g$, $a$, $A$ and $S$ are the same as in the proof of~\cite[Theorem~4.1]{KaRo83}. We can replace the inequality $g(A\cap S)\ge a-\frac52$ on~\cite[Theorem~4.1, p.~810]{KaRo83} by a stronger $g(A\cap S)\ge a-\frac32$ using~\eqref{nuestimate} for $g$ as follows:
\[
g(A \cap S) \ge g(A)+g(S)-g(A\cup S)-1 \ge \left(a-\frac12\right)+a-a-1=a-\frac32.
\]
We used $g(A)\ge a-\frac12$, $g(S)=a$, and $g(A\cup S)\le a$. Consequently, we can replace $\frac92$ by $\frac72$ everywhere in the proof of~\cite[Theorem~4.1]{KaRo83}. Accordingly, if $(6m,4m,3m,\gamma)$-concentrators exist for large enough $m$, then
\[
\wt K\le\frac{5+4\gamma-4/3}{2/3}.
\]
Hence, with $\gamma=5.05$, we obtain $\wt K\le 35.8<36$.
\end{proof}

\begin{proof}[Proof of Theorem~\ref{Brudnyi}]
We can assume that
\be\label{modulus}
\max_{x,y\in[0,1]^d}|f(x)+f(y)-2f((x+y)/2)|=\frac12,
\ee
and prove that for some linear polynomial $L$ we have $|f(x)-L(x)|\le \frac{73}2$, $x\in[0,1]^d$.

Let $\mA$ be the algebra of all subsets of $\{1,2,\dots,d\}$. Each element of $\mA$ can be naturally assigned to exactly one element of $\{0,1\}^d$ (the set of all vertices of the cube $[0,1]^d$) as follows. For any $A\in\mA$, let $\tau(A)=(x_1,\dots,x_d)$, where $x_j=1$ if $j\in A$, and $x_j=0$ otherwise. For any $f\in C([0,1]^d)$, we define a mapping $\nu:\mA\to\R$ as $\nu(A)=f(\tau(A))-f(0)$, $A\in\mA$. Under the assumption~\eqref{modulus}, we first claim that~\eqref{nuestimate} holds. Indeed, it is easy to see that
\[
\hat x := \frac{\tau(A)+\tau(B)}2=\frac{\tau(A\cap B)+\tau(A \cup B)}2\in[0,1]^d,
\]
so by~\eqref{modulus},
\begin{align*}
|\nu(A)+\nu(B)-\nu(A \cap B)-\nu(A \cup B)| &= |f(\tau(A))+f(\tau(B))-f(\tau(A \cap B))-f(\tau(A \cup B))| \\
&\le |f(\tau(A))+f(\tau(B))-2f(\hat x)|\\
&\quad + |f(\tau(A \cap B))+f(\tau(A \cup B))-2f(\hat x)|\\
&\le \frac12+\frac12=1.
\end{align*}
Applying Lemma~\ref{better}, we obtain an additive set-function $\mu$ satisfying $|\nu(A)-\mu(A)|\le 36$ for all $A\in\mA$. Note that by additivity of $\mu$, the linear function
\[
\wt L(x_1,\dots,x_d):=\mu(\{1\})x_1+\dots +\mu(\{d\})x_d
\]
satisfies $\wt L(\tau(A))=\mu(A)$, for any $A \in\mA$. Therefore, for the linear polynomial $L$ defined as $L(x):=\wt L(x)+f(0)$, we have the following estimate at the vertices of the cube:
\[
|f(x)-L(x)|\le 36, \quad x\in \{0,1\}^d.
\]
Now we show that this implies the required estimate for all $x\in[0,1]^d$. Let
\[
|f(\tilde x)-L(\tilde x)|=\max_{x\in[0,1]^d}|f(x)-L(x)|.
\]
Without loss of generality, assume that $\tilde x\in[0,\frac12]^d$ (otherwise we replace $0$ in the arguments below by an appropriate vertex of the cube). Since $2\tilde x\in[0,1]^d$, we use~\eqref{modulus} and $L(0)+L(2\tilde x)-2L(\tilde x)=0$ to conclude that
\begin{align*}
2|f(\tilde x)-L(\tilde x)| & \le |f(2\tilde x)-L(2\tilde x)| + |f(0)-L(0)| + |f(0)+f(2\tilde x)-2 f(\tilde x)|\\
&\le  |f(\tilde x)-L(\tilde x)|+36 +\frac 12.
\end{align*}
Hence, $|f(\tilde x)-L(\tilde x)|\le\frac{73}2$, as required.
\end{proof}


\section{Proof of Lemma~\ref{lemma-tech}}\label{tech}

We need to prove~\eqref{main}, which is
\[
\sum_{k=1}^{3m} \sum_{l=0}^{k} \sum_{r=0}^{k} \binom sl \binom{6m-s}{k-l} \binom {s-4m}{r} \binom{8m-s}{k-r} \frac{\binom{8k-r}{6k-l}}{\binom{36m-s}{6k-l}} =: \sum_{k=1}^{3m} \sum_{l=0}^{k} \sum_{r=0}^{k} a(m,s,k,l,r)<1.
\]
Let us give an outline of the proof. The main idea is to show that $a(m,s,k,l,r)\le e^{-cm}$ for some $c>0$. This will imply the required bound for large $m$, because there are at most $Cm^3$ terms of summation. We begin with relating binomial coefficients to a more convenient function $h(n,m)$ in Lemma~\ref{stirling}. Then we treat ``smaller'' values of $k$, i.e., $k\le \lceil 2.6m \rceil$, in Lemma~\ref{lemma2}. This case is easier, since there is a simple estimate for $\ds\sum_{l=0}^{k} \sum_{r=0}^{k} a(m,s,k,l,r)$ such that the bounding function (of $k$) attains maximum at the boundary of the domain. For the remaining more difficult case $\lceil 2.6m \rceil < k \le 3m$, we reduce the problem to optimization of a certain function $\phi$, as described in Lemma~\ref{maxphinegative}. First, we show analytically that $\phi$ attains its maximum when $k$ is largest. Then we show that the largest value of $\phi$ over the remaining two variables $l$ and $r$ will be attained at the only critical point of the domain, which is a solution of an algebraic system of equations of degree $5$. Numerical computations are used to verify the required conclusion on the maximum value of $\phi$.

Denote $g(x):=x\ln x$, if $x>0$, and $g(0):=g(0+)=0$. Let $h(x,y):=g(x)-g(y)-g(x-y)$. Note that $h$ is defined and continuous on $\{(x,y):0\le y\le x\}$, and also
\begin{equation}\label{homog}
h(\lambda x,\lambda y)=\lambda h(x,y), \quad \lambda>0.
\end{equation}
The following lemma relates the binomial coefficient $\binom nm$ with $h(n,m)$.
\begin{lemma}\label{stirling}
For any integer $n\ge1$ and $0\le m\le n$,
\[
\frac{1}{5\sqrt{n}} \exp(h(n,m)) \le  \binom nm  \le \exp(h(n,m)).
\]
\end{lemma}
\begin{proof}
Stirling's formula gives that for $n\ge1$
\[
\ln(n!)=\ln(\sqrt{2\pi})+n\ln n+\frac12\ln n -n + r(n),
\]
where $0<r(n)<\frac1{12n}$. This immediately implies the required estimates.
\end{proof}

Now we estimate the required sum when $k$ is not large.

\begin{lemma}\label{lemma2}
There is an integer $m_0$ such that for any integers $m\ge m_0$ and $s\le 6m$, we have
\be\label{eq1} \sum_{k=1}^{\lceil 2.6m\rceil} \sum_{l=0}^{k} \sum_{r=0}^{k} \binom
sl \binom{6m-s}{k-l} \binom {s-4m}{r} \binom{8m-s}{k-r}
\frac{\binom{8k-r}{6k-l}}{\binom{36m-s}{6k-l}}<\frac12. \ee
\end{lemma}
\begin{proof}
For simplicity, let $q=q(m):=\lceil 2.6m\rceil$. Since
$$
\sum_{l=0}^{k}\binom sl \binom{6m-s}{k-l}=\binom{6m}{k},
$$
and
$$
\sum_{r=0}^{k}\binom {s-4m}{r} \binom{8m-s}{k-r}=\binom{4m}{k},
$$
it is enough to prove that
$$
\sum_{k=1}^{q}\binom{6m}{k}\binom{4m}{k}\frac{\binom{8k}{5k}}{\binom{30m}{5k}}<\frac12.
$$
Using Lemma~\ref{stirling}, for $k\le q<3m$, we obtain
\begin{equation}
\label{eq2}
\binom{6m}{k}\binom{4m}{k}\frac{\binom{8k}{5k}}{\binom{30m}{5k}}\le
5\sqrt{30m} \exp(f(k,m)),
\end{equation}
where
\[
f(k,m):=h(6m,k)+h(4m,k)+h(8k,5k)-h(30m,5k).
\]
We have
$$
\frac{\partial^2}{\partial k^2}f(k,m)=\frac 3k+ \frac 4{6m-k}-\frac 1{4m-k}>0, \quad k\in (0,3m].
$$
Therefore the maximum of the right hand side in~\eqref{eq2} is attained for
$k=1$ or $k=q$. Hence,
\begin{equation}\label{locmax}
\sum_{k=1}^{q}\binom{6m}{k}\binom{4m}{k}\frac{\binom{8k}{5k}}{\binom{30m}{5k}}
< 15\sqrt{30}m^{3/2} \bigl(\exp(f(1,m))+\exp(f(q,m))\bigr).
\end{equation}
It is easy to see that $\ds\lim_{m\to\infty}m^3\exp(f(1,m))=C$,
for some $C>0$, hence $\ds\lim_{m\to\infty}m^{3/2}\exp(f(1,m))=0$. Also, by~\eqref{homog} and continuity of $h$,
\[
\lim_{m\to\infty}\frac{f(a,m)}{m}=h(6,2.6)+h(4,2.6)+h(8\cdot 2.6,5\cdot 2.6)-h(30,5\cdot 2.6)<-0.07,
\]
and so $\ds\lim_{m\to\infty}m^{3/2} \exp(f(q,m))=0$. Therefore, the limit of the right hand side of~\eqref{locmax} is zero as $m\to\infty$, hence, it is smaller than $\frac 12$ for large enough $m$, as required.
\end{proof}

The estimate of the remaining terms of~\eqref{main} will be deduced from an optimization problem, which we will describe now. The idea is to use Lemma~\ref{stirling} and~\eqref{homog} to establish asymptotics of each term of the required sum.

Let
\begin{equation}\label{phi}
\phi(c,k,l,r):=h(c,l)+h(6-c,k-l)+h(c-4,r)+h(8-c,k-r)+h(8k-r,6k-l)-h(36-c,6k-l).
\end{equation}
Clearly, for $c=5.7$ and $k\in[2.6,3]$ the above function $\phi$ is
defined when
\begin{equation}\label{dom}
k+c-6\le l\le k \quad\text{and}\quad k+c-8\le r\le c-4.
\end{equation}

Our optimization problem is described in the next lemma.

\begin{lemma}\label{maxphinegative}
The absolute maximum value of $\phi$ for $c=5.7$ and any $k\in[2.6,3]$ over all $l$ and $r$ given by~\eqref{dom} is a negative number.
\end{lemma}
\begin{proof}[Proof of Lemma~\ref{maxphinegative}.]
We claim that the absolute maximum of $\phi$ for $c=5.7$ and $k\in[2.6,3]$ over $l$ and $r$ given by~\eqref{dom} is achieved when $k=3$.
To simplify exposition, we will often present computations for a general fixed $c$ first, and then substitute $c=5.7$ in the end.

Observe that under the change of variables
\[
x=k-l,\quad y=\frac{c-4-r}{4-k},
\]
the inequalities~\eqref{dom} can be rewritten as
\[
0\le x\le 6-c \quad\text{and}\quad 0\le y\le 1.
\]
Therefore, we only need to show that for any fixed $x$, $y$ specified above, we have
\[
\frac{\partial \phi(c,k,x,y)}{\partial k}\ge0, \quad k\in[2.6,3].
\]
It is straightforward to compute that
\begin{align*}
\frac{\partial \phi(c,k,x,y)}{\partial k} &=
(\ln(c-k+x)-\ln(k-x))\\
&\quad +y(\ln((4-k)y)-\ln(c-4-(4-k)y))\\
&\quad +(1-y)(\ln(4-4y-k(1-y))-\ln(k(1-y)+4+4y-c))\\
&\quad +\Bigl[(8-y)\ln((8-y)k+4+4y-c)\\
&\quad-(3-y)\ln((3-y)k+4+4y-c-x)-5\ln(36-c-5k-x)\Bigr]\\
&=:D_1(c,k,x)+D_2(c,k,y)+D_3(c,k,y)+D_4(c,k,x,y).
\end{align*}
Many intermediary estimates below directly follow from monotonicity of the logarithm and bounds on the involved variables. We have
\begin{align*}
D_2(c,k,y)&=\frac{(c-4)}{(4-k)}\frac{(4-k)y}{(c-4)}
\left[\ln\left(\frac{(4-k)y}{(c-4)}\right)-\ln\left(1-\frac{(4-k)y}{(c-4)}\right)\right]\\
&\ge\frac{(c-4)}{(4-k)}\frac{(4-k)y}{(c-4)}
\left[\ln\left(\frac{(4-k)y}{(c-4)}\right)\right]\ge\frac{-(c-4)}{e(4-k)},
\end{align*}
where we used the fact that $\min_{t\in(0,1]}t\ln t=-1/e$. Similarly,
we get
\begin{align*}
D_3(c,k,y)&=\frac{(8-c)}{(4-k)}\frac{(4-k)(1-y)}{(8-c)}
\left[\ln\left(\frac{(4-k)(1-y)}{(8-c)}\right)-
\ln\left(1-\frac{(4-k)(1-y)}{(8-c)}\right)\right]\\
&\ge\frac{(8-c)}{(4-k)}\frac{(4-k)(1-y)}{(8-c)}\left[\ln\left(\frac{(4-k)(1-y)}{(8-c)}\right)\right]
\ge\frac{-(8-c)}{e(4-k)}.
\end{align*}
Clearly $D_1(c,k,x)\ge D_1(c,3,0)$, and similarly $D_4(c,k,x,y)\ge D_4(c,k,0,y)$.
With fixed $c$ and $k$, we claim that $D_4(c,k,0,y)$ attains minimum at $y=1$. Indeed,
\begin{align*}
\frac{\partial D_4(c,k,0,y)}{\partial y}&=\ln\left(1-\frac{5k}{8k+4-c+(4-k)y}\right)\\
&\quad +\frac{5(4-k)(4-c+4y)}{(3k+4-c+(4-k)y)(8k+4-c+(4-k)y)}\\&=:S_1(c,k,y)+\frac{S_2(c,k,y)}{S_3(c,k,y)}
\le S_1(5.7,2.6,1) + \frac{S_2(5.7,2.6,1)}{S_3(5.7,2.6,0)}\\
&=\ln\frac{15}{41}+\frac{16.1}{116.51}<0.
\end{align*}
Hence,
\begin{align*}
D_4(c,k,0,y)\ge D_4(c,k,0,1)
&=5\ln\left(\frac{7k+8-c}{36-c-5k}\right)+2\ln\left(1+\frac{5k}{2k+8-c}\right)\\
&=:T_1(c,k)+T_2(c,k)\ge T_1(5.7,2.6)+T_2(5.7,2.6)\\
&=5\ln\frac{20.5}{17.3}+2\ln\frac{20.5}{7.5}>2.
\end{align*}
In summary,
\[
\frac{\partial \phi(c,k,x,y)}{\partial k}\ge D_1(c,3,0)-\frac{(c-4)}{e(4-k)}-\frac{(8-c)}{e(4-k)}+2= \ln\frac{2.7}{3}-\frac4e+2>0,
\]
so $\phi(5.7,k,x,y)\le \phi(5.7,3,x,y)$, and we can now focus on the case $k=3$.

With $c=5.7$ and $k=3$ the restrictions~\eqref{dom} become $l\in[2.7,3]$ and $r\in[0.7,1.7]$. To find the critical points of $\phi$ inside the domain we compute the partial derivatives of $\phi$:
\begin{align}
\label{dl}
\frac{\partial\phi(c,3,l,r)}{\partial l} &=\ln\left(\frac{(c-l)(3-l)(18-c-l)}{l(3-c+l)(6-r+l)}\right),\\
\label{dr}
\frac{\partial\phi(c,3,l,r)}{\partial r} &=\ln\left(\frac{(c-4-r)(3-r)(6-r+l)}{r(5-c+r)(24-r)}\right).
\end{align}
The system of equations $\{\frac{\partial \phi}{\partial l}=0,\ \frac{\partial \phi}{\partial r}=0\}$ can be reduced to the following algebraic equation of degree $5$ on $l$:
\begin{gather}
(2c-18)l^5+(-2c^2-69c+846)l^4+(-2c^3+123c^2+189c-11448)l^3 \nonumber \\
+(2c^4+12c^3-2349c^2+14256c+95256)l^2+(-48c^4+1089c^3+2916c^2-125388c)l  \label{eqonl}\\
+126c^4-4536c^3+40824c^2=0. \nonumber
\end{gather}
This reduction and some further computations were performed using Maple software\footnote{A copy of the corresponding Maple worksheet is available at {\tt http://prymak.net/concentrators.pdf}}. When $l$ is found, $r$ can be obtained from $\frac{\partial \phi}{\partial l}=0$, which is a linear equation on $r$. This allows us to compute all critical points numerically with any given precision. In particular, for $c=5.7$, we get that there is only one critical point $(l^*,r^*)\in(2.7,3)\times(0.7,1.7)$, and it satisfies \[
|l^*-\bar l|<10^{-7}, \quad |r^*-\bar r|<10^{-7},
\]
where $(\bar l, \bar r)=(2.8959102,1.078108)$ is an approximate numerical solution.

We want to prove that the value of $\phi$ at the critical point is negative, that is $\phi(5.7,3,l^*,r^*)<0$. At the approximation of the critical point we have $\phi(5.7,3,\bar l,\bar r)<-0.004$, so it suffices to show that $\phi$ cannot change much around our point, more precisely, we need
\[
|\phi(5.7,3,\bar l,\bar r)-\phi(5.7,3,l^*,r^*)|<0.004.
\]
This can be done by estimating the partial derivatives of $\phi$ in a rectangle that contains both $(l^*,r^*)$ and $(\bar l, \bar r)$, say in $[2.89,2.9]\times[1.07,1.08]$. Rewriting~\eqref{dl} and~\eqref{dr} as sums of logarithms, using monotonicity of the logarithm and the restrictions $l\in[2.89,2.9]$ and $r\in[1.07,1.08]$, it is straightforward to show that
\[
\left|\frac{\partial\phi}{\partial l}\right|<10 \quad\text{and}\quad
\left|\frac{\partial\phi}{\partial r}\right|<10.
\]
Therefore, as required,
\[
|\phi(5.7,3,\bar l,\bar r)-\phi(5.7,3,l^*,r^*)|<20\cdot 10^{-7}<0.004.
\]
We proved that $\phi$ is negative at the only critical point inside the domain $[2.7,3]\times[0.7,1.7]$.

It remains to show that $\phi$ cannot achieve its maximum on the boundary of $[2.7,3]\times[0.7,1.7]$. Indeed, from~\eqref{dl}, it is easy to see that for any fixed $r\in(0.7,1.7)$, we have
\[
\lim_{l\to 2.7^+}\frac{\partial\phi(5.7,3,l,r)}{\partial l}=+\infty, \quad\text{and}\quad
\lim_{l\to 3^-}\frac{\partial\phi(5.7,3,l,r)}{\partial l}=-\infty.
\]
Similar arguments apply to $\frac{\partial\phi}{\partial r}$, for a fixed $l\in(2.7,3)$. This completes the proof of the lemma.
\end{proof}

Finally, we are ready for a formal proof of the required estimate.

\begin{proof}[Proof of Lemma~\ref{lemma-tech}]
In view of Lemma~\ref{lemma2}, we only need to show that
\[
\sum_{k=\lceil 2.6m\rceil+1}^{3m} \sum_{l=0}^{k} \sum_{r=0}^{k} \binom
sl \binom{6m-s}{k-l} \binom {s-4m}{r} \binom{8m-s}{k-r}
\frac{\binom{8k-r}{6k-l}}{\binom{36m-s}{6k-l}}<\frac12.
\]
Each term of the sum can be estimated by Lemma~\ref{stirling} as follows:
\[
\binom
sl \binom{6m-s}{k-l} \binom {s-4m}{r} \binom{8m-s}{k-r}
\frac{\binom{8k-r}{6k-l}}{\binom{36m-s}{6k-l}}
<30\sqrt{m}\exp(\psi(m,s,k,l,r)),
\]
where
\begin{align*}
\psi(m,s,k,l,r):=&h(s,l)+h(6m-s,k-l)+h(s-4m,r)+h(8m-s,k-r)\\
&+h(8k-r,6k-l)-h(36m-s,6k-l).
\end{align*}
Recalling that $s=s(m)=\lceil 5.7m\rceil$, $h$ is continuous, and using~\eqref{homog}, we see that
\[
\lim_{m\to\infty}\frac{\psi(m,s,k,l,r)}m=\phi(5.7,k,l,r).
\]
According to Lemma~\ref{maxphinegative}, $\phi(5.7,k,l,r)\le-\delta$, for some $\delta>0$. Therefore, for large enough $m$ and some $\delta_1>0$, we have
\[
\sum_{k=\lceil 2.6m\rceil+1}^{3m} \sum_{l=0}^{k} \sum_{r=0}^{k} \binom
sl \binom{6m-s}{k-l} \binom {s-4m}{r} \binom{8m-s}{k-r}
\frac{\binom{8k-r}{6k-l}}{\binom{36m-s}{6k-l}}<0.4\cdot3^2\cdot 30 m^{7/2} e^{-\delta_1 m},
\]
which tends to zero as $m\to\infty$, and so the required sum is smaller than $\frac12$ for large enough $m$.
\end{proof}

\begin{remark}\label{c**}
Denote by $c^*$ the supremum of all $c$ such that the statement of
Lemma~\ref{maxphinegative} remains true. One can prove that
$c^*$ is the unique solution of the equation
$$
\phi(c,3,l(c),r(c))=0, \quad c\in[5.7,6],
$$
where $\phi$ is given by~\eqref{phi}, and $l=l(c)\in[2.7,3]$ and $r=r(c)\in[0.7,1.7]$ is the solution of the system
$\{\frac{\partial \phi}{\partial l}=0,\ \frac{\partial \phi}{\partial r}=0\}$, see~\eqref{dl}, \eqref{dr}. More detailed numerical computations show that $c^*\in(5.724889,5.72489)$.
Hence, the maximum value of $s=s(m)$ for which~\eqref{main} holds
satisfies $\lim\limits_{m\to\infty}s(m)/m=c^*$. For simplicity, we stated and proved the lemma for $c=5.7$, as the
optimal value $c^*$ provides only slight improvement to the constants in
Section~\ref{sec-const}.
\end{remark}

\end{document}